\documentclass[a4,12pt]{article}

\usepackage{amsmath,amsfonts,amsthm,latexsym,amssymb,mathrsfs}
\usepackage{eucal}

\def\R{\mathcal{R}}
\def\B{\mathcal{B}}
\def\A{\mathcal{A}}
\def\K{\mathcal{K}}
\newtheorem{df}{Definition}[section]
\newtheorem{thm}[df]{Theorem}
\newtheorem{cor}[df]{Corollary}
\newtheorem{rema}[df] {Remark}
\newtheorem{lem}[df] {Lemma}

\newcommand{\mmm}[4]
{\left[ \begin{array}{cc} #1 & #2 \\ #3 & #4 \end{array} \right]}

\begin{document}

\title{\bf The inverse along an element in \\ rings  with  an involution, Banach \\ algebras
and $C^*$-algebras}
\author{Julio  Ben\'{\i}tez and Enrico \rm Boasso}\date{ }

\maketitle

\begin{abstract}\noindent Properties of the inverse along an element in rings with an involution, 
	Banach algebras and $C^*$-alegbras will be studied unifying known expressions concerning
	generalized inverses.\par
\par
\medskip
\noindent {\bf AMS Subjects Classifications:} 15A09; 16W10; 46H99; 46L99
\end{abstract}

\section{Introduction}

There exist many specific generalized inverses in the literature, such as the group inverse,
the Drazin inverse and the Moore-Penrose inverse. 
Recently  X. Mary have unified  these different notions of invertibility in \cite{mary} by introducing a new type 
of outer inverse. Furthermore, several authors have studied this new outer inverse (see \cite{mary, M2, mary1, marybis, KM, ZCP, ZPC}).

Some properties of this latter pseudoinverse were studied in \cite{bb} in the setting
of rings. In this article further properties will be studied enlarging the underlying set.
Specifically, rings with an involution, Banach algebras,
and $C^*$-algebras will be consider. The main objective of this article is to study some properties of the Mary inverse such as 
limits, representations and continuity and the relationship between the aforementioned inverse and the weighted Moore-Penrose inverse.

\section{Preliminary definitions and  facts}

From now on, $\R$ will denote a unitary ring with unity $1$. Let $\R^{-1}$ be the set of
invertible elements of $\R$. Given $a \in \R$, we define
the {\em image ideals} by $a \R = \{ ax: x \in \R \}$ and $\R a = \{ xa : x \in \R \}$, and 
the {\em kernel ideals} by $a^{-1}(0) = \{ x \in \R: ax=0 \}$ and $a_{-1}(0) = \{ x \in \R: xa=0 \}$.

An element $a \in \R$ is said to be {\em regular} if there exists $x \in \R$ such that $a=axa$.
The element $x$, which is not uniquely determined by $a$, will be said to be an {\em inner inverse}
of $a$. The set of regular elements of $\R$ will be denoted by $\widehat{\R}$.
If $y \in \R$ satisfies $yay=y$, then it will be said that $y$ is an {\em outer inverse} of $a$.

Next the definition of the key notion of this article will be recalled.

\begin{df}\label{inversa_de_mary}
Let $\R$ be a ring with unity and consider $a,d \in \R$. The element $a$
is said to be {\em invertible along} $d$ if, there exists $b \in \R$ such that
$bab=b$ and $b\R=d\R$, $\R b = \R d$. 
\end{df}

In the conditions of Defintion \ref{inversa_de_mary}, according to
\cite[Theorem 6]{mary}, if such $b$ exists, then it is unique. This element $b$ satisfying
the conditions of Defintion \ref{inversa_de_mary} will be said to be
{\em the inverse of $a$ along $d$} and it will be denoted by $a^{\parallel d}$. 
Moreover, according to \cite[p. 3]{marybis}, if $a^{\parallel d}$ exists, then $d$ is regular. 

Note that according to \cite[Theorem 3.3]{bb}, if $a$ is invertible
along $d$, then $d_{-1}(0) = b_{-1}(0)$ and 
$d^{-1}(0) = b^{-1}(0)$, where $b=a^{\parallel d}$. Hence, from
$bab=b$, it can be easily proved that
\begin{equation}\label{bad}
a^{\parallel d} a d = d = da a^{\parallel d}.
\end{equation}

\noindent In addition, a straightforward calculation proves that $a \in \R$ is invertible along $1$ if and only if $a\in\R^{-1}$.

Next follow the definitions of several classical pseudoinverses which are particular cases
of the inverse along an element.

The element $a \in \R$ is said to be {\em group invertible}, if there exists
$b \in \R$ such that
$$
aba=a, \qquad bab=b, \qquad ab=ba.
$$
If $a$ is group invertible, then such element $b$ it is unique and it is customary written 
$a^\#$. According to \cite[Theorem 11]{mary}, $a$ is group invertible if and only
if $a$ in invertible along $a$. The set of all group invertible elements of $\R$ will be denoted by $\R^\#$.

The element $a \in \R$ is said to be {\em Drazin invertible}, if there exists
$b \in \R$ such that
$$
a^mba=a, \qquad bab=b, \qquad ab=ba,
$$
for some $m \in \mathbb N$. 
If $a$ is Drazin invertible, then such element $b$ is unique and it is customary written 
$a^{d}$. According to \cite[Theorem 11]{mary}, $a$ is Drazin invertible if and only
if $a$ in invertible along $a^k$ for some $k \in \mathbb N$.
The set of all Drazin invertible elements of $\R$ will be denoted by $\R^d$.

In this paragraph it will be assumed that $\R$ has an involution. Recall that
an involution $*:\R \to \R$ is an operation that satisfies 
$$
(a+b)^* = a^*+b^*, \qquad (ab)^*=b^*a^*, \qquad (a^*)^* = a,
$$
for all $a, b\in \R$. The element $a \in \R$ is 
said to be {\em Moore-Penrose invertible}, if there exists
$b \in \R$ such that
$$
aba=a, \qquad bab=b, \qquad (ab)^*=ab, \qquad (ba)^*=ba.
$$
If $b$ is Moore-Penrose invertible, then such element $b$ is unique and it is customary written 
$a^\dag$. According to \cite[Theorem 11]{mary}, $a$ is Moore-Penrose invertible if and only
if $a$ in invertible along $a^*$.
The set of all Moore-Penrose invertible elements of $\R$ will be denoted by $\R^\dag$.

Finally, recall that if $p \in \R$ is an idempotent (i.e., $p^2=p$), then it is easy to prove that 
$p \R p$ is a subring of $\R$ whose unity is $p$. If $\mathcal{S}$ is a subring
of $\R$ and $x \in \mathcal{S}^{-1}$, then $(x)^{-1}_\mathcal{S}$ will denote the inverse of $x$ in the subring $\mathcal{S}$. 

\section{Rings with an involution}

Let $\R$ be a unitary ring with an involution. 
An element $a \in \R$ is said to be Hermitian if $a^*=a$.
Evidently, if $a\in \R^{-1}$, then $a^* \in \R^{-1}$ and 
$(a^*)^{-1} = (a^{-1})^*$. 

According to Definition \ref{inversa_de_mary}, the next result is obvious.

\begin{rema}\label{remark1}
	Let $\R$ be a unitary ring with an involution and $a,d \in \R$. Then
	$a$ is invertible along $d$ if and only if $a^*$ is invertible along $d^*$. In this
	case, $(a^{\parallel d})^* = (a^*)^{\parallel d^*}$.
\end{rema}

The former result encloses the following (\cite[Theorem 11]{mary}). If $a \in \R$, then 
\begin{itemize}
	\item[(i)] $a \in \R^\# \hbox{ \rm if and only if } a^* \in \R^\#$. In this case, $(a^\#)^* = (a^*)^\#$.
            \item[(ii)]  $a \in \R^d  \hbox{ \rm if and only if } a^* \in \R^d$. In this case, $(a^d)^* = (a^*)^d$.
	\item[(iii)] $a \in \R^\dag  \hbox{ \rm if and only if } a^* \in \R^\dag$. In this case, $(a^\dag)^* = (a^*)^\dag$.
\end{itemize}

To the best knowledge of the authors, the weighted Moore-Penrose inverse is a generalised inverse which has not been linked 
to the inverse along an element yet. Next the relationship between the aforementioned inverses will be study. In first place,
the definition of the weighted Moore-Penrose inverse will be recalled. \par

Let $m$ and  $n \in \R$ be invertible and Hermitian. Then given $a \in \R$, the set of elements $x \in \R$ such that
\begin{equation}\label{mp_con_pesos}
axa=a, \qquad xax=x, \qquad (max)^*=max, \qquad (nxa)^*=nxa
\end{equation}
is empty or a singleton. In order to show that 
no extra hypotheses on $m$ and $n$ are necessary, the proof of the uniqueness will be given. If $x$ and $y$ satisfies
(\ref{mp_con_pesos}), then 
$$max = mayax=(may)m^{-1}(max),$$
 which, by taking $*$, implies that
$$max = (max)m^{-1}(may) = maxay = may.$$ 
Hence $ax=ay$. In a similar way it is possible to prove that
$xa=ya$. However, $x=xax=xay=yay=y$.  
When the set under consideration is a singleton, 
$a$ will be said to be {\em weighted Moore-Penrose invertible relative to $m$ and $n$} and the unique element
satisfying (\ref{mp_con_pesos}) will be said to be the  {\em weighted Moore-Penrose inverse of $a$ relative to $m$ and $n$};
in addition, it will be denoted by $a^\dag _{m,n}$.

In order to link the invertibility along an element and the weighted Moore-Penrose inverse, 
some extra assumptions on $m$ and $n$ are needed, namely, the positivity. Next this notion will be recalled.\par
The element $x \in \R$ will be said to be {\em positive}, if there exists
a Hermitian $y \in \R$ such that $x=y^2$. In this case, the element  $y$ will be said to be a {\em square root} of $x$. \par	

Observe that in a $C^*$-algebra, 
every positive element has a unique square root. For arbitrary rings, this is not true. Take
$\R = \mathbb{Z}_6$. Since $\R$ is commutative, then $\R$ has an involution, namely, the 
identity, 
and therefore any element in $\mathbb{Z}_6$ is Hermitian.
In addition, $[2]^2 = [4]$ and $[4]^2 = [4]$, which implies that $[4]$ has two square
roots. 

If $\R$ is a ring with an involution, $x \in \R$ is positive, and $y$ a square root
of $x$, then it is easy to see that $x$ is Hermitian and if $x$ is invertible, then 
$y$ is also invertible. In fact, since $y$ is Hermitian by definition,  $x^* = (y^2)^* = 
(y^*)^2 = y^2=x$. In addition, since $x=y^2$, $xy=y^3=yx$. Now, since
$x\in\R^{-1}$, $x^{-1}y=yx^{-1}$. However, since
$1=xx^{-1} = y(yx^{-1})$ and $1=x^{-1}x = (x^{-1}y)y$, $y$ is invertible.

In the following theorem the relationship between the inverse along an element and the weighted Moore-Penrose
inverse will be presented.

\begin{thm}\label{thm10000}
Let $\R$ be a unitary ring with an involution and consider  $a \in \R$
and two invertible and positive element $m,n \in \R$. The following 
statements are equivalent.
\begin{itemize}
\item[{\rm (i)}] $a$ is weighted Moore-Penrose invertible relative to $m$ and $n$.
\item[{\rm (ii)}] $a$ is invertible along $n^{-1}a^*m$.
\end{itemize}
Furthermore, in this case, $a^{\parallel n^{-1}a^*m} = a^\dag _{m,n}$.
\end{thm}
\begin{proof} Suppose that statement (i) holds and denote $x = a^\dag _{m,n}$. Since
$$
x=xax=n^{-1}nxax = n^{-1}(nxa)^*x = n^{-1}a^* mm^{-1}x^* n^* x, 
$$
$x\R \subset n^{-1} a^* m \R$. Observe now that 
$(nxa)^* = nxa$ implies $n^{-1}a^* x^* = xan^{-1}$ ($n$ is invertible 
and Hermitian). Hence
$$
n^{-1}a^*m = n^{-1} (axa)^* m = n^{-1} a^* x^* a^* m = xan^{-1} a^* m,
$$
which leads to $n^{-1} a^* m \R \subset x \R$.
Thus, it has been proved that $n^{-1} a^* m \R = x \R$. The proof of 
$\R n^{-1} a^* m = \R x$ follows from 
$$
x = xax = xm^{-1}max = xm^{-1}(max)^* = xm^{-1}x^*nn^{-1}a^*m 
$$
and 
$$
n^{-1}a^*m = n^{-1}(axa)^*m = n^{-1}a^*x^*a^*m = n^{-1}a^*(max)^* = n^{-1}a^*max.
$$

Now suppose that statement (ii) holds and denote $y = a^{\parallel n^{-1}a^*m}$.
According to Definition \ref{inversa_de_mary},  $yay=y$. By (\ref{bad}),
$yan^{-1}a^*m = n^{-1}a^*m = n^{-1}a^*m ay$, or equivalently, 
\begin{equation}\label{aux1}
yan^{-1}a^* = n^{-1}a^*, \qquad a^*m = a^*m ay.
\end{equation}
Let $p, q \in \R$ be square roots of $m$ and $n$, respectively. Observe that
by definition, $p$ and $q$ are Hermitian. Furthermore, 
$p$ and $q$ are invertible since $m$ and $n$ are invertible.
Note that from the second equality of (\ref{aux1}),  
$a^* p = a^*may p^{-1}$, and by the involution
\begin{equation}\label{aux2}
pa = p^{-1} (may)^* a = p^{-1}(ppay)^*a = p^{-1}(pay)^*pa
= (payp^{-1})^*pa.
\end{equation}
Thus, $payp^{-1} =  (payp^{-1})^*payp^{-1}$. In particular, $payp^{-1}$ is Hermitian.
Since $may = p(payp^{-1})p$, $may$ is Hermitian. In addition, since $payp^{-1}$ is Hermitian and $p$ is invertible, according to (\ref{aux2}),  it is possible to conclude that $aya=a$.

It remains to prove that $nya$ is Hermitian. To this end, consider now
the first equality of (\ref{aux1}), which is equivalent to 
$qyan^{-1}a^*=q^{-1}a^*$. Hence
$$
aq^{-1} = a(yan^{-1})^*q = a(yaq^{-1}q^{-1})^*q = aq^{-1}(yaq^{-1})^*q
= aq^{-1}(qyaq^{-1})^*.
$$
Thus, $qyaq^{-1}=qyaq^{-1}(qyaq^{-1})^*$. As a result, $qyaq^{-1}$ is Hermitian.
Since $nya = q(qyaq^{-1})q$, $nya$ is Hermitian.
\end{proof}

The second statement of Theorem \ref{thm10000} leads to a characterization
of the weighted Moore-Penrose inverse by means of invertible elements.
Note first that, if $\R$ is a unitary ring with an involution and $a\in\R$, then $a$ is regular
if and only if $a^*$ is regular.

\begin{thm}Let $\R$ be a unitary ring with an involution and consider  $a \in \widehat{\R}$
and two invertible and positive element $m,n \in \R$. If $z$ is any inner inverse of $a^*$, then the following 
statements are equivalent.\par
\begin{itemize}
\item[{\rm (i)}] $a$ is weighted Moore-Penrose invertible relative to $m$ and $n$.
\item[{\rm (ii)}] $u=a^*man^{-1} +1 -a^*z$ is invertible.
\item[{\rm (iii)}] $v=man^{-1}a^* + 1-za^*$ is invertible.
\end{itemize}
\noindent In this case, 
$$
a^\dag _{m,n}=a^{\parallel n^{-1}a^*m}=n^{-1}u^{-1}a^*m=n^{-1}a^*v^{-1}m.
$$
\end{thm}
\begin{proof}Apply Theorem \ref{thm10000} and \cite[Theorem 2.3]{ZPC}
(or \cite[Corollary 3.8]{ZCP}).
\end{proof}

Let $\R$ be a ring, $a \in \R$ and $s \in \R^{-1}$. It is very simple to prove that 
$a \in \R^\#$ if and only if $s^{-1}as \in \R^\#$ and in this case $(s^{-1}as)^\# = s^{-1}a^\#s$.
This expression is specially useful in matrix theory to investigate
the group inverse of particular matrices since frecuently a matrix $B$ can be decomposed
as $S^{-1}AS$, where $S$ is non-singular and $A$ is simpler than $B$ (e.g. when 
$B$ is diagonalisable).
However, if the ring has an involution and $a \in \R^\dag$, 
it is not in general true that $(s^{-1}as)^\dag = s^{-1}a^\dag s$ (which shows 
that the spectral decomposition of a matrix, in general, cannot be used to
investigate the Moore-Penrose inverse of a matrix).

A related discussion is the following. Let $\R$ be a ring with an involution, $a \in \R$
and $u,v \in \R$ unitary (i.e., $u^{-1}=u^*$ and $v^{-1}=v^*$).
Then, necessary and sufficient for $a \in \R^\dag$ is that  $uav^* \in \R^\dag$; moreover, in this case,
$(u^*av)^\dag = v^*a^\dag u$. 
Again, this expression is used in matrix theory, since
the singular value decomposition of a square complex matrix $B$ allows 
to write $B=U^* \Sigma V$, where $\Sigma$ is diagonal (and real), and $U,V$ are 
unitary. 
However, in general, the expression $(u^*av)^\# = v^*a^\#u$ does not hold; which
shows that the singular value decomposition is not useful to find the group
inverse of a matrix.

Next  two results generalizing and completing the discussion of the 
two previous paragraphs will be presented.

\begin{thm} \label{th28j}
Let $\R$ be a unitary ring and consider $a,d\in \R$ such that $a$ is invertible along $d$.
If $s,r \in \R^{-1}$, then $sar^{-1}$ is invertible along  $rds^{-1}$ and
$(sar^{-1})^{\parallel rds^{-1}} = r a^{\parallel d} s^{-1}$.
\end{thm}
\begin{proof} Three facts must be proved: $\R ra^{\parallel d}s^{-1}=\R rds^{-1}$,
$ ra^{\parallel d}s^{-1} \R= rds^{-1}\R$, and $ra^{\parallel d}s^{-1}$ is 
an outer inverse of $sar^{-1}$. These facts can be deduced  from $\R a^{\parallel d}=\R d$,
$a^{\parallel d} \R= d\R$, and $a^{\parallel d}$ is an outer inverse of $a$, 
respectively. 
\end{proof}

\begin{rema}\rm Note that according to Theorem \ref{th28j}, if
$\R$ is a unitary ring with an involution and $u,v$ are unitary, 
then $uav^*$ is invertible along $vdu^*$ and
$(uav^*)^{\parallel vdu^*} = v a^{\parallel d} u^*$, where  $a$ and $d$ are as in Theorem \ref{th28j}.\end{rema}

Let $\R$ be a ring
with an involution and consider $a\in \R^\dag$ and $s,r \in \R^{-1}$.
According to Theorem \ref{th28j} and \cite[Theorem 11]{mary}, $sar^{-1}$ is invertible along $ra^*s^{-1}$ and 
$(sar^{-1})^{\parallel ra^*s^{-1}} = ra^\dag s^{-1}$.
If  $c=sar^{-1}$, then $c$ is invertible along
$r(s^{-1}cr)^*s^{-1}$ and $c^{\parallel r(s^{-1}cr)^*s^{-1}} = r(s^{-1}cr)^\dag s^{-1}$.
Observe that $r(s^{-1}cr)^*s^{-1} = rr^*c^*(ss^*)^{-1}$.
Thus,  the following theorem has been partially proved.

\begin{thm}
Let $\R$ be a unitary ring with an involution and consider $a\in \R$ and $s, r  \in \R^{-1}$.
Suppose that $s^{-1}ar \in \R^\dag$.
\begin{itemize}
\item[{\rm (i)}]  The element $a$ is invertible along $rr^*a^*(ss^*)^{-1}$ and
$(s^{-1}ar)^\dag = r^{-1}a^{\parallel rr^*a^*(ss^*)^{-1}} s$.
\item[{\rm (ii)}] If $a \in \R^\dag$, then $(s^{-1}ar)^\dag = r^{-1}a^\dag s$
if and only if $ss^*a\R = arr^*\R$ and $\R ss^*a = \R a rr^*$.
\end{itemize}
\end{thm}
\begin{proof}Statement (i) was proved in the paragraph preceding this Theorem.\par

To prove statement (ii), define $f=rr^* a^* (ss^*)^{-1}$. To deduce statement (ii)
from statement (i), it is necessary to prove that necessary and sufficient for $a^\dag = a^{\parallel f}$
is that $ss^*a\R = arr^*\R$ and $\R ss^*a = \R a rr^*$. According to the definition of the inverse along
an element, $a^\dag = a^{\parallel f}$ if and only if 
$\R a^\dag = \R f$ and  $a^\dag \R = f\R$.
Recall that $\R a^\dag = \R a^*$, hence $\R a^\dag = \R f$ is equivalent to 
$a\R = f^* \R$. Since $f^* = (ss^*)^{-1} a rr^*$, 
$a\R = f^* \R$ if and only if $ss^* a \R = a rr^* \R$. The remaining identity can be prove in a similar way.
\end{proof}

The following  result deals with an expression of $(s^{-1} a r)^\#$, where $r$ and $s$ are 
invertible elements of a unitary ring $\R$.

\begin{thm} \label{thm28k}Let $\R$ be a unitary ring and consider 
$a, u, v \in \R$ such that $r$ and $s$ are invertible.
If $s^{-1}ar \in \R^\#$, then
\begin{itemize}
\item[{\rm (i)}] $a$ is invertible along $rs^{-1}ars^{-1}$ and $(s^{-1}ar)^\# = 
r^{-1} a^{\parallel rs^{-1} a rs^{-1}} s$.
\item[{\rm (ii)}] If $a \in \R^\#$, then $(s^{-1}ar)^\# = r^{-1}a^\# s$ if and only if 
$\R asr^{-1} = \R rs^{-1}a$ and $sr^{-1}a \R = ars^{-1} \R$.
\end{itemize} 
\end{thm}
\begin{proof}(i). 
Since $s^{-1}ar \in \R^\#$, $s^{-1}ar$ is invertible along $s^{-1}ar$ 
(\cite[Theorem 11]{mary}). According to
Theorem \ref{th28j}, $a=s(s^{-1}ar)r^{-1}$ is invertible along $rs^{-1}ars^{-1}$ and
$a^{\parallel rs^{-1}ars^{-1}} = r(s^{-1}ar)^\#s^{-1}$.

\noindent (ii). Let $g=rs^{-1}ars^{-1}$. 
According to statement (i), $(s^{-1}ar)^\# = r^{-1}a^\# s$ if and only if
$a^{\parallel g} = a^\#$. Recall that  $a^\#$ is an outer inverse of $a$, $\R a^\# = \R a$, and $a^\# \R = a\R$.
As a result,
$$
a^{\parallel g} = a^\# \iff 
\begin{cases} \R a^\# = \R g \\  a^\# \R = g \R \end{cases} \iff 
\begin{cases} \R a = \R rs^{-1}ars^{-1} \\  a \R = rs^{-1}ars^{-1} \R \end{cases} \iff
\begin{cases} \R asr^{-1} = \R rs^{-1}a \\  sr^{-1}a \R = ars^{-1} \R. \end{cases} 
$$
\end{proof}

\indent In the particular case of rings with an involutiom, the following result can be deduced.

\begin{cor}Let $\R$ be a unitary ring with an involution and consider $a, u, v \in \R$ such that $u$ and $v$ are unitary.
If $u^*av \in \R^\#$, then
\begin{itemize}
\item[{\rm (i)}] $a$ is invertible along $vu^*avu^*$ and $(u^*av)^\# = 
v^* a^{\parallel vu^* a vu^*} u$.
\item[{\rm (ii)}] If $a \in \R^\#$, then $(u^*av)^\# = v^* a^\# u$ if and only if 
$\R auv^* = \R vu^*a$ and $uv^*a \R = avu^* \R$.
\end{itemize} 
\end{cor}
\begin{proof}Apply Theorem \ref{thm28k}.
\end{proof}

There is a matrix representation for elements in unitary rings which has been useful
to prove many results in the previous literature. Next follows this representation.
Let $p \in \R$ be an idempotent. Any element $x$ in a unitary ring $\R$ can be represented as
follows:
\begin{equation}\label{matriz}
x = \mmm{pxp}{px(1-p)}{(1-p)xp}{(1-p)x(1-p)}.
\end{equation}
Observe that 
$$
x = pxp + px(1-p) + (1-p)xp + (1-p)x(1-p).
$$
Recall that since $p$ is an idempotent, $p \R p$ and $(1-p) \R (1-p)$ 
are subrings with units $p$ and $1-p$, respectively.
If in addition, $\R$ has an involution and the idempotent $p$ is Hermitian, then
the above matrix representation preserves the involution, i.e., 
$$
\mmm{x_1}{x_2}{x_3}{x_4}^* = \mmm{x_1^*}{x_3^*}{x_2^*}{x_4^*}.
$$

If $a$ and $d$ are elements in a unitary ring with an involution
and $a$ is invertible along $d$, then
$a^{\parallel d}$ is an outer inverse of $a$. Hence 
$aa^{\parallel d}$ and $a^{\parallel d}$ are idempotents. It will be characterized when these idempotents are Hermitian.
Before doing this, an useful representation will be given. To this end,
however, first recall that according to \cite[Theorem 3.1]{bb},
$a$ is invertible along $d$ if and only if $dap$ is invertible in
$p\R p$, where $p=dd^-$ and $d^-$ is such that $d=dd^-d$. Furthermore, in this case, $a^{\parallel d}=wd$, where 
$w = (dap)_{p \R p}^{-1}$.

\begin{lem}\label{lemarep}
Let $\R$ be a unitary ring and consider $a\in\R$ and $d \in \widehat{\R}$. If $d^-$ is an inner inverse of $d$ and 
if the representation of $a$ respect the idempotent $p=dd^-$ is
$$
a = \mmm{x}{y}{z}{t},
$$
then
$$
d = \mmm{dp}{d(1-p)}{0}{0}, \qquad
da = \mmm{dap}{da(1-p)}{0}{0}, \qquad ap=x+z.
$$
Furthermore, if $a$ is invertible along $d$, then

\begin{equation}\label{matrixaII}
a^{\parallel d} =  (dap)_{p \R p}^{-1}d=\mmm{(dap)_{p \R p}^{-1}dp}{(dap)_{p \R p}^{-1}d(1-p)}{0}{0}.
\end{equation}
\end{lem}
\begin{proof} The representation of $d$ can be deduced from the fact that
$pd=d$. The representation of $da$ is evident.
Since 
$$
ap=\mmm{x}{0}{z}{0},
$$
\noindent  $ap=x+z$. Recall that according to \cite[Theorem 3.1]{bb},
$a^{\parallel d}=(dap)_{p \R p}^{-1}d$. In particular,
the representation of  $a^{\parallel d}$ is evident.
\end{proof}

\begin{thm}\label{t29}Let $\R$ be a unitary ring with an involution. Let $a\in \R$ and $d \in \widehat{\R}$ such that
$a$ is invertible along $d$. 
\begin{itemize}
\item[{\rm (i)}] If there exists $d^-$, an inner inverse of $d$, such that
$dd^-$ Hermitian, then $a^{\parallel d} a$ is Hermitian if and only if 
$(da)^* \in d\R$.
\item[{\rm (ii)}] If there exists $d^-$, an inner inverse of $d$, such that
$d^- d$ Hermitian, then $a a^{\parallel d}$ is Hermitian if and only if 
$ad \in d^* \R$.
\end{itemize}
\end{thm}
\begin{proof}(i). Let $p$ be the Hermitian idempotent $p=dd^-$.
The matrix representation given in Lemma \ref{lemarep} will be used.
Denote also $w = (dap)^{-1}_{p \R p}$.
\begin{equation*}
\begin{split}
a^{\parallel d}a & = \mmm{wdp}{wd(1-p)}{0}{0}
\mmm{x}{y}{z}{t} = \mmm{wdx+wdz}{wdy+wdt}{0}{0} \\
& = \mmm{wdap}{wda(1-p)}{0}{0} = \mmm{p}{wda(1-p)}{0}{0}.
\end{split}
\end{equation*}
Since $p$ is Hermitian, the above matrix representation preserves the
involution. Therefore, $a^{\parallel d}a$ is Hermitian if and only if  $wda(1-p)=0$, which 
is equivalent to $wda=wdap$, which in turn is equivalent to $pda=pdap$ (recall that $w$ is the inverse of
$dap$ in $p\R p$). However,  $da=pda=pdap=dadd^-$.\par
Now, if $da=dadd^-$, then $(da)^* = (da(dd^-))^* = dd^- (da)^* \in d\R$.
If $(da)^* \in d\R$, then $(da)^* = du$ for some $u \in \R$, and then
$dadd^- = u^*d^*dd^- = u^*d^*(dd^-)^* = u^*(dd^-d)^* = u^*d^* = da$. 

(ii). Apply statement  (i) and Remark \ref{remark1}.
\end{proof}

\section{A representation of the inverse along an element}

\noindent In this section a representation of the inverse along an element will be presented. First,
however, two facts need to be recalled. Given a group invertible element
$x$ in a unitary ring $\R$, the {\em spectral idempotent} of $x$ is defined
as $x^\pi = 1-xx^\#$. In addition, recall that if $a\in \R$ and $d\in\widehat{\R}$
are such that $a$ is invertible along $d$, then $ad$ and $da$ are group invertible
(\cite[Theorem 7]{mary}).

\begin{lem}\label{lgi} Let $\R$ be a unitary ring and consider $a\in \R$ and $d \in \widehat{\R}$. If $a$ is invertible along $d$,
$d^-$ is an inner inverse of $d$ and $p=dd^-$, then, using the representation 
in Lemma \ref{lemarep}, 
$$
(da)^\# = \mmm{w}{w^2da(1-p)}{0}{0}, \qquad (da)^\pi = \mmm{0}{-wda(1-p)}{0}{1-p},
$$
where $w$ is the inverse of $dap$ in $p \R p$.
\end{lem}
\begin{proof}Using the matrix representation of $da$ given in 
Lemma \ref{lemarep} it easy to see that
$$
u = \mmm{w}{w^2da(1-p)}{0}{0}
$$
is the group inverse of $da$.
In addition, 
\begin{equation*}
\begin{split}
(da)^\pi & = 1-(da)(da)^\# \\ & = 
\mmm{p}{0}{0}{1-p} - \mmm{dap}{da(1-p)}{0}{0} \mmm{w}{w^2da(1-p)}{0}{0}
= \mmm{0}{-wda(1-p)}{0}{1-p}.
\end{split}
\end{equation*}
\end{proof}

\indent In the following theorem, a representation of the inverse along an element will be proved.

\begin{thm} \label{thm1000}
Let $\R$ be a unitary ring. Let $a\in\R$ and $d \in \widehat{\R}$ be such that $a$ is invertible along
$d$ and consider  $d^-$,  an inner inverse of $d$, and $p=dd^-$. For $t\in\R$, necessary and sufficient
for $da+t (da)^\pi$ to be invertible is that
$-(1-p)twda(1-p)+ (1-p)t(1-p)$ is invertible in the subring $(1-p)\R (1-p)$,
where $w= (dap)^{-1}_{p\R p}$. Moreover, under
this situation,
$$
a^{\parallel d} = (da+t (da)^\pi)^{-1}d, 
$$
\end{thm}
\begin{proof} Consider $t\in\R$ and represent it using the idempotent $p$, i.e., 
$$
t=\mmm{t_1}{t_2}{t_3}{t_4}.
$$
According to the matrix representations given in Lemma \ref{lemarep} and Lemma \ref{lgi},
\begin{align*}
da+t(da)^\pi &= \mmm{dap}{da(1-p)}{0}{0} + \mmm{t_1}{t_2}{t_3}{t_4}\mmm{0}{-wda(1-p)}{0}{1-p}\\
                   &= \mmm{dap}{da(1-p)-t_1wda(1-p)+ t_2}{0}{-t_3wda(1-p)+ t_4}.\\
\end{align*}
Since $dap$ is invertible in $p \R p$ (\cite[Theorem 3.1]{bb}),
$da+t(da)^\pi$ is invertible if and only if
$-t_3wda(1-p)+ t_4$ is invertible in the subring $(1-p) \R (1-p)$.
In addition, under this situation, 
$$
(da+t(da)^\pi)^{-1} = \mmm{w}{\xi}{0}{\mu},
$$
for some $\xi$ and $\mu \in \R$. Now, using the representations of  $d$ and $a^{\parallel d}$ presented in Lemma \ref{lemarep},
$$
(da+t(da)^\pi)^{-1}d = \mmm{w}{\xi}{0}{\mu} \mmm{dp}{d(1-p)}{0}{0}
= \mmm{wdp}{wd(1-p)}{0}{0} = a^{\parallel d}.
$$
\end{proof}

\indent In the case of an algebra, Theorem \ref{thm1000} particularizes as follows.
Note that if  $\K$ is a field and $\A$ is a  $\K$-algebra, then given $t\in\K$,
$tz=(t.1)z$, where $z\in\A$ and 
$1$ stands for the unit of $\A$.\par

\begin{thm}\label{thm2000}Let $\K$ be a field and consider a $\K$-algebra $\A$.
Let $a\in\A$ and $d \in \widehat{\A}$ be such that $a$ is invertible along
$d$ and consider  $t\in\K$, $t\neq 0$. Then,
$$
a^{\parallel d} = (da+t (da)^\pi)^{-1}d.
$$
\end{thm}
\begin{proof}Let $d^-$ be an inner inverse of $d$ and let $p=dd^-$. 
As in  Theorem \ref{thm1000}, consider $w=(dap)^{-1}_{p\R p}$. 
Since $(1-p)w=0$,
$$-(1-p)twda(1-p)+ (1-p)t(1-p)= -t(1-p)wda(1-p) +t(1-p)=t(1-p).
$$
Therefore, to conclude the proof, apply Theorem \ref{thm1000}.
\end{proof}

\begin{rema}\label{remark3000}\rm Let $\R$ be a unitary ring and consider $a\in\R$ and $d\in \widehat{\R}$ be such that
$a$ is invertible along $d$. Recall that according to \cite[Theorem 7]{mary}, $da$ and $ad$ are group invertible.
Note that the results presented in   Theorem \ref{thm1000} and Theorem \ref{thm2000}
concerns the element $da$. However, considering the ring $(\R, +, \diamond)$, where $a\diamond b= ba$,  it is possible to prove similar results to the ones presented in the aforementioned Theorems considering the element $ad$. In fact, it is evident that $a$ is invertible along $d$ (respectively group invertible) in $\R$  if and only if $a$ is invertible along $d$ (respectively group invertible) in $(\R, +, \diamond)$. For example, under the same hypotheses in Theorem \ref{thm2000},
it is possible to conclude that 
$$
a^{\parallel d}= d(ad+t (ad)^\pi)^{-1}.
$$ 
The details are left to the reader.
\end{rema}

\section{Inverses along an element and limits}

In this section, several results concerning inverses along an element
and limits will be proved. \par

Throughout this section $\B$ will denote a Banach algebra or a $C^*$-algebra.
If $x\in\B$, then $\sigma (x)$ will stand for the spectrum of $x$. Note that the matrix representation (\ref{matriz}) with respect to the
idempotent $p\in\B$ also preserves  limits. In other words, 
if $(x_n)_{n\in\mathbb{N}}\subset\B$ and $x \in B$ are represented as
$$
x_n = \mmm{a_n}{b_n}{c_n}{d_n}, \qquad x=\mmm{a}{b}{c}{d},
$$
respectively, then it is not difficult to prove that
$(x_n)_{n\in\mathbb{N}}$ converges to  $x$ if and only if
$(a_n)_{n\in\mathbb{N}}$ (respectively $(b_n)_{n\in\mathbb{N}}$, $(c_n)_{n\in\mathbb{N}}$,
$(d_n)_{n\in\mathbb{N}}$) converges to $a$ (respectively $b$, $c$, $d$).
 
\indent Next the inverse along an element will be presented as a limit.\par

\begin{thm}\label{thm31}Let $\B$ be a Banach algebra and consider $a\in\B$ and $d\in\widehat{\B}$ such that $a$ is 
invertible along $d$. Then,
\begin{itemize}
\item[{\rm (i)}] $\lim_{t \to 0}(da + t 1)^{-1}d$ exists and it equals to $a^{\parallel d}$,
\item[{\rm (ii)}] $\lim_{t \to 0}d(ad + t 1)^{-1}$ exists and it equals to $a^{\parallel d}$.
\end{itemize}
\end{thm} 
\begin{proof}(i). Let $d^-$ be an inner inverse of $d$ and consider $p=dd^-$. If $a$ and $d$ are represented 
as in Lemma \ref{lemarep}, then, given $t \in \mathbb{R}$,  
\begin{equation}\label{da+t1}
da+t1 = \mmm{dap+tp}{da(1-p)}{0}{t(1-p)}.
\end{equation}
Since $a$ is invertible along $d$, according to \cite[Theorem 7]{mary},
$da$ is group invertible. In particular, $0$ is an isolated point of  $\sigma (da)$
(\cite[Theorem 4]{King}). Then, there exists $U \subset \mathbb{C}$,
a punctured neighbourhood of 0, such that $da+t1 \in \B^{-1}$ 
for each $t \in U$. Hence, according to the representation of $da+t1$ presented above, 
$dap+tp \in (p \B p)^{-1}$ for  each $t \in U$.
Denote by $w_t$ the inverse of $dap+tp$ in $p \B p$ ($t\in U$).  Then, for each $t\in U$, there exists $\xi_t\in\B$ such that
$$
(da+t1)^{-1}d = \mmm{w_t}{\xi_t}{0}{t^{-1}(1-p)}
\mmm{dp}{d(1-p)}{0}{0} = \mmm{w_t dp}{w_t d(1-p)}{0}{0} = w_td.
$$
Thus, according to (\ref{matrixaII}), to prove the Theorem, it 
is enough to prove that $\lim_{t \to 0} w_t$ exists
and $\lim_{t \to 0} w_t = (dap)^{-1}_{p \B p}$. 

But these affirmations follow from: 
a) $w_t$ is the inverse  of $dap+tp$ in $p \B p$ and
b) $p \B p $ is a Banach algebra and 
the standard inverse is a continuous map from 
$\mathcal{G}(p\B p)$ to $\mathcal{G}(p\B p)$, where
$\mathcal{G}(p\B p)$ is the set of invertibles in $p\B p$.\par

\indent (ii). Apply Remark \ref{remark3000} to the Banach algebra $(\B, +, \diamond )$ and  statement (i).
\end{proof}

Next some special cases will be considered.
\begin{cor}Let $\B$ be a Banach algebra and consider $a \in \B$. 
\begin{itemize}
\item[{\rm (i)}] If $a$ is group invertible, then $\lim_{t \to 0}(a^2+t1)^{-1}a$ and $\lim_{t \to 0}a(a^2+t1)^{-1}$ exist and both limits equal to $a^\#$.
\item[{\rm (ii)}] If $a$ is Drazin invertible with ${\rm ind}(a) = k$, then 
$\lim_{t \to 0}(a^{k+1}+t1)^{-1}a^k$ and $\lim_{t \to 0}a^k(a^{k+1}+t1)^{-1}$ exist and both limits equal to $a^d$.
\item[{\rm (iii)}] If $\B$ is a $C^*$-algebra and $a$ is Moore-Penrose invertible, then 
$\lim_{t \to 0}(a^*a+t1)^{-1}a^*$ and $\lim_{t \to 0}a^*(aa^*+t1)^{-1}$ exist and both limits equal to $a^\dagger$.
\end{itemize}
\end{cor}
\begin{proof}Apply \cite[Theoem 11]{mary} and Theorem \ref{thm31}
\end{proof}

To prove the next result, it will be useful to previously establish
a simple bound. Let $\B$ be a Banach algebra and let 
$a,b \in \B$ be invertible elements. Then,
\begin{equation*}
a^{-1}-b^{-1} = b^{-1}(b-a)a^{-1} = 
(b^{-1} - a^{-1})(b-a)a^{-1} + a^{-1}(b-a)a^{-1},
\end{equation*}
which implies
$$
\| a^{-1}-b^{-1} \| \leq 
\| b^{-1} - a^{-1} \| \| b-a \| \| a^{-1} \| + 
\| a^{-1} \|^2 \| b-a\| ,
$$
or equivalenty (if $\| b-a \| \| a^{-1} \| < 1$)
\begin{equation}\label{cota}
\| a^{-1} - b^{-1} \| \leq \frac{\| a^{-1} \|^2 \| b-a\|}
{1-\| b-a \| \|a^{-1}\|}. 
\end{equation}

Since to prove the following result an involution is needed, $C^*$-algebras
will be considered. Recall that according to \cite[Theorem 6]{HM}, given $d\in\A$,
$\A$ a $C^*$-algebra, necessary and sufficient for $d$ to be regular is that $d$ is
Moore-Penrose inversible. Thus, in this case, $d$ has  an inner inverse
$d^-$ such that $dd^-$ is Hermitian.  

\begin{thm}\label{teorema31}Let $\A$  be a $C^*$-algebra and consider $a\in\A$ and $d \in \widehat{\A}$ such 
that $a$ is invertible along $d$. Let $d^-\in\A$ be such that $d^-$ is an inner inverse of $d$  and $dd^-$ is Hermitian.
Then, for enough small $t$,
$$
\| (da+t1)^{-1}d-a^{\parallel d} \| \leq 
\frac{t \| a^{\parallel d}\|^2 \| d^- \|^2 }
{1-t \| a^{\parallel d}\| \| d^- \|} \| d \|.
$$
\end{thm}
\begin{proof} 
Let $t > 0$, $p=dd^-$, and denote by $w$ the inverse of $dap$ in $p\B p$ (\cite[Theorem 3.1]{bb}).
According to  the proof of Theorem \ref{thm31},
there exists $U \subset \mathbb{C}$,
a punctured neighbourhood of 0, such that $da+t1 \in \B^{-1}$ 
for each $t \in U$. Moreover,  if $w_t$ denotes the inverse of $dap+tp$ in $p \B p$ ($t\in U$) (see Theorem \ref{thm31}), then, according 
to  Lemma \ref{lemarep} and the proof of Theorem \ref{thm31},
$$
(da+t1)^{-1}d-a^{\parallel d} = 
\mmm{(w_t -w) dp}{(w_t - w)d(1-p)}{0}{0} = (w_t-w)dp.
$$
Since $p=dd^-$ is a Hermitian idempotent, $\|p\|=1$, and thus,
$$
\| (da+t1)^{-1}d-a^{\parallel d} \| \leq \| w_t-w \| \| d\|.
$$
Note that from (\ref{matrixaII}),  $a^{\parallel d} = wd$. Hence
$a^{\parallel d}d^- = wdd^- = wp = w$. Thus, $\| w \| \leq \| a^{\parallel d}\| \|d^-\|$.
Now, since $w$ and $w_t$ are invertible in $p\R p$ ($t\in U$), if 
$t \in U$ and $t\| a^{\parallel d}\| \| d^- \| < 1$, then from (\ref{cota}) can be deduced that
$$
\| w - w_t \| \leq \frac{t \| w\|^2 }{1-t\| w\|} \leq 
\frac{t \| a^{\parallel d}\|^2 \| d^- \|^2 }
{1-t \| a^{\parallel d}\| \| d^- \|}.
$$
\end{proof}

Considering the Banach algebra $(\B, +, \diamond )$, the following Corollary can be deduced.  
Note  that according to \cite[Theorem 6]{HM}, given $d\in\widehat{\A}$,
$\A$ a $C^*$-algebra, $d$ has  an inner inverse
$d^-$ such that $d^-d$ is Hermitian.

\begin{cor}Let $\A$  be a $C^*$-algebra and consider $a\in\A$ and $d \in \widehat{\A}$ such 
that $a$ is invertible along $d$. Let $d^-\in\A$ be such that $d^- $ is an inner inverse of $d$  and $d^- d$ is Hermitian.
Then, for enough small $t $,
$$
\|d (ad+t1)^{-1}-a^{\parallel d} \| \leq 
\frac{t \| a^{\parallel d}\|^2 \| d^- \|^2 }
{1-t \| a^{\parallel d}\| \| d^- \|} \| d \|.
$$
\end{cor}
\begin{proof}
Apply Remark \ref{remark3000} to the Banach algebra $(\B, +, \diamond )$ and Theorem \ref{thm31}.
\end{proof}

Next it will be characterized when $\lim_{t \to 0}(da + t1)^{-1}f$ exists, where $f\in\B$.

\begin{thm}\label{thm6000}Let $\B$ be a Banach algebra and consider $a$, $d$ and $f \in \B$. Assume that $d$ is regular.
\begin{itemize}
\item[{\rm (i)}] If $0$ is not a limit point of $\sigma (da)$ and  $\lim_{t \to 0}(da+t1)^{-1}f$ exists, then $f \in d\B$. 
\item[{\rm (ii)}] If $f \in d\B$ and $a$ is invertible along $d$, then 
there exists $\lim_{t \to 0}(da+t1)^{-1}f$ and it
equals to $a^{\parallel d}d^-f$, where $d^-$ is an arbitrary
inner inverse of $d$.
\end{itemize}
\end{thm}
\begin{proof}Recall that if $a$ is invertible along $d$, then according to \cite[Theorem 7]{mary}, $da$ is group invertible. Thus,
according to \cite[Theorem 4]{King}, $0$ is not a limit point of $\sigma (da)$. Therefore, according to the hypotheses of statements (i) and (ii),
there exists $U$, a punctured neighbourhood of $0$ such that $dap+t1 \in \B^{-1}$ ($t\in U$). Let $d^-$ be an inner inverse of $d$ and let $p=dd^-$. In addition, represent $a$ and $d$ as in Lemma \ref{lemarep}. Hence, as in the proof of Theorem \ref{thm31}, $dap+tp \in  (p\B p)^{-1}$ for $t \in U$.
Let $w_t$ be the inverse of $dap+tp$ in $p \B p$. According to  
(\ref{da+t1}), there exist $\xi_t \in p\B (1-p)$ such that, for each $t\in U$,
$$
(da+t1)^{-1} = \mmm{w_t}{\xi_t}{0}{t^{-1}(1-p)}.
$$
In particular, 
\begin{equation}\label{daf}
(da+t1)^{-1}f = w_t f + \xi_t f + t^{-1}(1-p)f. 
\end{equation}

(i). Since  $\lim_{t \to 0}(da+t1)^{-1}f$ exists, 
$\lim_{t \to 0} (1-p)(da+t1)^{-1}f$ exists. But (\ref{daf}) implies that
$(1-p)(da+t1)^{-1}f = t^{-1}(1-p)f$. Therefore, 
$(1-p)f=0$, equivalently $f = pf = dd^-f \in d\B$.

(ii). Since $f \in d\B$, $pf = f$ and $(1-p)f=0$. Hence, (\ref{daf}) implies that
$(da+t1)^{-1}f = w_t f$. Since $a$ is invertible along $d$, 
$dap$ is invertible in $p \B p$ (\cite[Theorem 3.1]{bb}).  
As in the proof of Theorem \ref{thm31}, the continuity of the standard inverse in $p\B p$
implies that  $\lim_{t \to 0}w_t = (dap)^{-1}_{p \B p}$.
However, according to (\ref{matrixaII}),
$$
a^{\parallel d}d^- = (dap)^{-1}_{p \B p}dd^- = 
(dap)^{-1}_{p \B p}p = (dap)^{-1}_{p \B p}.$$
\end{proof}

\begin{rema}\rm
If $a,d,f$ are elements in a Banach algebra such that 
$0$ is not a limit point of $\sigma(da)$ and $\lim_{t \to 0}(da+t1)^{-1}f$ exists, then it is not possible to conclude
that that $a$ is invertible along $d$. For example, take $f=0$.	
\end{rema}

The following Theorem will state the symmetric version of Theorem \ref{thm6000}.

\begin{thm}Let $\B$ be a Banach algebra and consider $a$, $d$ and $f \in \B$. Assume that $d$ is regular.
\begin{itemize}
\item[{\rm (i)}] If $0$ is not a limit point of $\sigma (ad)$ and  $\lim_{t \to 0}f(ad+t1)^{-1}$ exists, then $f \in \B d$. 
\item[{\rm (ii)}] If $f \in \B d$ and $a$ is invertible along $d$, then 
$\lim_{t \to 0}f(ad+t1)^{-1}$ exists and it equals to $fd^-a^{\parallel d}$, where 
$d^-$ is an arbitrary inner inverse of $d$.
	\end{itemize}
\end{thm}
\begin{proof}Apply Remark \ref{remark3000} to the Banach algebra $(\B, +, \diamond)$ and Theorem \ref{thm6000}.

\end{proof}

The following representation extends the one given in \cite{show}. 
The case of the Moore-Penrose inverse in $C^*$-algebras was studied in \cite[Example 3.6]{koliha}.

\begin{thm}\label{ext_show}Let $\B$ be a Banach algebra and consider $a\in \B$  and $d\in \widehat{\B}$. Let $d^-$ be 
any inner inverse of $d$ and $p=dd^-$. Then, If exists $\beta \in \mathbb{R} \setminus \{0\}$
such that $\| p - \beta dap \| < 1$, then
$a$ is invertible along $d$ and 
$$
a^{\parallel d} = \beta \sum_{n=0}^\infty (1-\beta da)^n d = \beta 
\sum_{n=0}^\infty d(1-\beta ad)^n.
$$
\end{thm}
\begin{proof} Since the second equality can be derived form the first one by symmetry, only the first equality will be proved.\par
Since $dap \in p\B p$ and $\| p - \beta dap \|< 1$, 
the element $\beta dap$ is invertible in $p \B p$ and the 
serie $\sum_{n=0}^{\infty}(p - \beta dap)^n$
converges to the inverse of $\beta dap$ in $p \B p$.
In particular, according to \cite[Theorem 3.1]{bb},
$a$ is invertible along $d$ and $( \beta dap )^{-1}_{p \B p}= \beta^{-1}( dap )^{-1}_{p \B p}$.\par

Next represent $a$ and $d$ as in Lemma \ref{lemarep}.
Since,
$$
1-\beta da = \mmm{p-\beta dap}{-\beta da(1-p)}{0}{1-p},
$$
  there exists
a sequence $(\xi_n)_{n\in\mathbb N} \subset p \B (1-p)$ such that 
$$
(1 - \beta da)^n = \mmm{(p-\beta dap)^n}{\xi_n}{0}{1-p}.
$$
Now, 
\begin{equation*}
\begin{split}
(1 - \beta da)^n d & = \mmm{(p - \beta dap)^n}{\xi_n}{0}{1-p} 
\mmm{dp}{d(1-p)}{0}{0} \\
& = \mmm{(p - \beta dap)^n dp }{(p - \beta dap)^n d(1-p) }{0}{0} = (p - 
\beta dap)^n d.
\end{split}
\end{equation*}
However, according to \cite[Theorem 3.1]{bb},
$$
\sum_{n=0}^\infty (1 - \beta da)^n d = 
\sum_{n=0}^\infty (p - \beta dap)^n d = 
\left( \beta dap \right)^{-1}_{p \B p} d = \beta^{-1} a^{\parallel d}.
$$
\end{proof}

It is well known that the Moore-Penrose inverse of a 
complex $m \times n$ matrix $A$ is given by $A^\dag = 
\int_0^\infty \exp(-A^*At) A^* {\rm d}t$ (for a very simple
proof, see \cite[Problem 73.2]{ams}). This integral representation was
extended to $C^*$-alegbras in \cite[Example 3.5]{koliha}.
There is a similar representation for the Drazin inverse
(see \cite{castroetalia}). The next Theorem will generalize these results.

\begin{thm}\label{thm7000}Let $\B$ be a Banach algebra and consider $a\in \B$ and $d\in\widehat{\B}$
such that $a$ is invertible along $d$ and $\sigma (da) \setminus \{0 \} \subset 
\{ z \in \mathbb{C} : {\rm Re }(z) > 0 \}$. Then
$$
\int_0^\infty \exp(-tda) \ d \ {\rm d}t = a^{\parallel d}.
$$
\end{thm}
\begin{proof}
It is known that if the non-zero spectrum of $x \in \B$ lies in the open right half 
of the complex plane, then
\begin{equation}\label{integral}
x^{-1} = \int_0^\infty \exp(-tx) {\rm d}t.
\end{equation}
Let $d^-$ be any inner inverse of $d$ 
and $p=dd^-$. Then,
\begin{equation*}
\exp(-tda) d = \sum_{k=0}^\infty \frac{(-tda)^k}{k!} d
= \sum_{k=0}^\infty \frac{(-t)^k (da)^k d}{k!}.
\end{equation*}
Now, according to Lemma \ref{lemarep}, there exists a sequence  $(\xi_k)_{k\in\mathbb N}\subset p\B (1-p)$, such that
$$
(da)^k d = \mmm{(dap)^k}{\xi_k}{0}{0} 
\mmm{dp}{d(1-p)}{0}{0} = \mmm{(dap)^kdp}{(dap)^kd(1-p)}{0}{0}
= (dap)^k d.
$$
Thus,
$\exp(-tda) d = \exp(-tdap)d$

The representation of $da$ given in Lemma \ref{lemarep}
implies that $\sigma (da) = \sigma_{p \A p}(dap) \cup \{ 0 \}$,
where $\sigma_{p \A p}(dap)$ stands for the spectrum of $dap$
in the Banach algebra $p\B p$.
The hypothesis on $\sigma(da)$ and the invertibility
of $dap$ in $p \B p$ (\cite[Theorem 3.1]{bb}), implies that 
$\sigma_{p \A p}(dap)\subset 
\{ z \in \mathbb{C} : {\rm Re }(z) > 0 \}$.
In particular, according to (\ref{integral}) and using $x=dap$ in the subalgebra $p \B p$, 
$$
\int_0^\infty \exp(-tdap) {\rm d}t = (dap)^{-1}_{p \A p}.$$

However, according to \cite[Theorem 3.1]{bb},
$$
\int_0^\infty \exp(-tda) d {\rm d}t = \left( \int_0^\infty 
 \exp(-tdap) {\rm d}t  \right) d = (dap)^{-1}_{p \A p} d = 
a^{\parallel d}.
$$
\end{proof}

\begin{rema}\rm As it has been done before, applying Remark \ref{remark3000} to 
$(\B, +, \diamond)$ and Theorem \ref{thm7000}, the following statement can be derived 
using a symmetric argument: Let $\B$ be a Banach algebra and consider $a\in \B$ and $d\in\widehat{\B}$
such that $a$ is invertible along $d$ and $\sigma (ad) \setminus \{0 \} \subset 
\{ z \in \mathbb{C} : {\rm Re }(z) > 0 \}$. Then
$$
\int_0^\infty d \exp(-tad) \  \ {\rm d}t = a^{\parallel d}.
$$
\end{rema}

\section{The continuity of the inverse along an element}

To prove the continuity of the inverse 
along an element, first  the following Lemma need to be proved.

\begin{lem}\label{bound}
Let $\B$ be a unitary Banach algebra and consider $a,b,d,e \in \B$ such that $a$ is invertible
along $d$ and $b$ is invertible along $e$. Let $d^-$ be an inner inverse of $d$ and
$e^-$ be a inner inverse of $e$.
Then
$$
b^{\parallel e} - a^{\parallel d} = 
b^{\parallel e}e^-(e-d)(1-a a^{\parallel d}) +
(1-b^{\parallel e}b)(e-d)d^-a^{\parallel d} + b^{\parallel e}(a-b)a^{\parallel d}.
$$
\end{lem}
\begin{proof}
From (\ref{bad}) and \cite[Theorem 3.1]{bb}, 
$b^{\parallel e}be = e$ and $dd^-a^{\parallel d} = a^{\parallel d}$.
Therefore, 
\begin{equation*}
\begin{split}
(1-b^{\parallel e}b)(e-d)d^-a^{\parallel d} & = 
\left[ (1-b^{\parallel e}b) e - (1-b^{\parallel e}b) d\right]d^- a^{\parallel d} \\
& = -(1-b^{\parallel e}b)dd^-a^{\parallel d} \\
& = b^{\parallel e}b a^{\parallel d}-a^{\parallel d}.
\end{split}
\end{equation*}
In a similar way it is possible to prove that
$$
b^{\parallel e}e^-(e-d)(1-a a^{\parallel d}) = b^{\parallel e} - b^{\parallel e}aa^{\parallel d}.
$$
Therefore, 
\begin{equation*}
\begin{split}
b^{\parallel e} - a^{\parallel d} & = 
b^{\parallel e}e^-(e-d)(1-a a^{\parallel d}) + b^{\parallel e}aa^{\parallel d}
+ (1-b^{\parallel e}b)(e-d)d^-a^{\parallel d} - b^{\parallel e}b a^{\parallel d} \\
& = b^{\parallel e}e^-(e-d)(1-a a^{\parallel d}) +
(1-b^{\parallel e}b)(e-d)d^-a^{\parallel d} + b^{\parallel e}(a-b)a^{\parallel d}.
\end{split}
\end{equation*}
\end{proof}

The next result deals with the continuity of the invertibility along an element.
For the special cases of the group inverse, the Drazin inverse and the Moore-Penrose
inverse, see \cite{koliha, RS}.

\begin{thm}Let $\B$ be a Banach algebra and consider two sequences $(a_n)_{n\in\mathbb N}\subset \B$ and $(d_n)_{n\in\mathbb N}\subset\B$ converging to 
$a$ and $d$, respectively. Suppose that $a$ is invertible along $d$ and $a_n$ is invertible along $d_n$, for each $n\in\mathbb N$.
Let $d_n^-$ be an inner inverse of $d_n$,  $n\in\mathbb N$. Then, the following statements hold.
\begin{itemize}
\item[{\rm (i)}] If $a_m^{\parallel d_m}$ converges to $a^{\parallel d}$, then $a_m^{\parallel d_m}$ is a bounded sequence.
\item[{\rm (ii)}] If $a_m^{\parallel d_m}$ is a bounded sequence and $\sup_n \| d_n^- \|<\infty$, then $a_m^{\parallel d_m}\to a^{\parallel d}$ .
\end{itemize}
\end{thm}
\begin{proof}Statement (i) is obvious. 
\noindent To prove statement (ii), apply Lemma \ref{bound}.
\end{proof}

\bigskip
\noindent Julio Ben\'{\i}tez\par
\noindent E-mail address: jbenitez@mat.upv.es \par
\medskip
\noindent Enrico Boasso\par
\noindent E-mail address: enrico\_odisseo@yahoo.it
\end{document}